\documentclass[a4paper, 10pt,leqno]{amsart}
\date{December 26, 2024}
\usepackage[dvipdfmx]{graphicx}
\usepackage{latexsym}
\usepackage{amssymb}
\usepackage{amsmath}
\usepackage{amscd}
\usepackage{amsthm}
\usepackage{amsfonts}
\usepackage{mathrsfs}
\usepackage{enumerate}
\usepackage{enumitem}
\usepackage{bm}
\usepackage[usenames]{color}
\usepackage[active]{srcltx}
\usepackage{comment}
\usepackage[bbgreekl]{mathbbol}
\usepackage{here}
\usepackage[subrefformat=parens]{subcaption}
  \DeclareSymbolFontAlphabet{\mathbb}{AMSb}
  \DeclareSymbolFontAlphabet{\mathbbl}{bbold}
  \DeclareMathSymbol{\bbepsilon}{\mathord}{bbold}{"0F}
\title{The total absolute curvature of submanifolds with singularities}

\author[Y.~Yamauchi]{Yuta Yamauchi}
\address{%
   Graduate School of Engineering Science, 
   Yokohama National University, 
   Hodogaya, Yokohama 240-8501, Japan
}
\email{yamauchi-yuta-hj@ynu.jp}


\usepackage{amsthm}
\theoremstyle{plain}
 \newtheorem{theorem}{Theorem}[section]
 \newtheorem{introtheorem}{Theorem}

 \newtheorem{proposition}[theorem]{Proposition}
 \newtheorem{fact}[theorem]{Fact}
 \newtheorem*{fact*}{Fact}
 \newtheorem{lemma}[theorem]{Lemma}
 \newtheorem{corollary}[theorem]{Corollary}
 \theoremstyle{remark}
 \newtheorem{definition}[theorem]{Definition}
 \newtheorem{remark}[theorem]{Remark}
 \newtheorem*{acknowledgements}{Acknowledgements}
 \newtheorem{example}[theorem]{Example}
\numberwithin{equation}{section}

 \makeatletter
    
    \@addtoreset{equation}{section}
  \makeatother

\newcommand{\bmath}[1]{\boldsymbol #1}

    \newcommand{\vect}[1]{\bmath{#1}}
\newcommand{\R}{{\boldsymbol R}}

\usepackage{comment}
\begin{document}

\begin{abstract}    
In this paper, we give a generalization of 
the Chern-Lashof theorem for submanifolds with singularities called frontals in Euclidean space.
We prove that, for an $n$-dimensional admissible compact frontal in $(n+r)$-dimensional Euclidean space $\R^{n+r}$,
its total absolute curvature is greater than or equal to the sum of the Betti numbers.
Furthermore, 
if the total absolute curvature is equal to $2$, and all singularities are of the first kind, then the image of the frontal coincides with a closed convex domain of an affine $n$-dimensional subspace of $\R^{n+r}$. 
\end{abstract}

\maketitle


\section{Introduction}

We fix positive integers $n$ and $r$.
We consider an oriented compact manifold $M^n$ of dimension $n$ 
and an immersion $f:M^n \to \R^{n+r}$ into $(n+r)$-dimensional Euclidean space $\R^{n+r}$.
We denote by $B$ the bundle of unit normal vectors, 
and by $G$ the Lipschitz-Killing curvature of $f$.
Then,
$$
\tau (M^n,f) = \frac{1}{\operatorname{vol}(S^{n+r-1})} \int_B |G|\, d \mu_B
$$
is called the
{\it total absolute curvature} of $f$, where $\operatorname{vol}(S^{n+r-1})$ denotes the volume of an $(n+r-1)$-dimensional unit sphere and $d \mu_B$ is the volume element of $B$.
The following holds:

\medskip

\begin{fact}[\cite{CL1,CL2}]\label{fa:CL}
Let $M^n$ be an oriented compact manifold of dimension $n$,
and let $f:M^n \to \R^{n+r}$ be an immersion into $(n+r)$-dimensional Euclidean space $\R^{n+r}$.
\begin{itemize}
\item[$(1)$]  
Let $b_i (M^n)$ be the $i$-th Betti number $(0 \leq i \leq n)$ of $M^n$.
Then, the total absolute curvature $\tau (M^n,f)$ of $f$ satisfies the inequality
$$
\tau(M^n,f) \geq \sum_{i=0}^n b_i (M^n).
$$
\item[$(2)$]  
If the total absolute curvature $\tau(M^n,f)$ is less than $3$, then $M^n$ is homeomorphic to an $n$-dimensional sphere.
\item[$(3)$]  
If the total absolute curvature $\tau(M^n,f)$ is equal to $2$,
then the image $f(M^n)$ is contained in an $(n+1)$-dimensional affine subspace of $\R^{n+r}$, 
and is embedded as a convex hypersurface. The converse of this is also true.
\end{itemize}
\end{fact}
\medskip

The Chern-Lashof theorem (Fact~\ref{fa:CL}) can be understood as a generalization of Fenchel's theorem \cite{Fenchel1929, Fenchel1951}
to compact Euclidean submanifolds of any dimension.
So far,  generalizations of the Chern-Lashof theorem for several ambient spaces have been obtained:
for submanifolds in Riemannian manifolds of non-positive curvature \cite{WS, BYC1}; 
for submanifolds in spaces of constant curvature \cite{BYC2, Teu1, Teu2, Teu3}; 
for non-closed submanifolds in Euclidean space \cite{Win}; 
for knotted surfaces \cite{KM}; 
for equiaffine immersions \cite{Koi1};
for submanifolds in simply connected symmetric spaces of non-positive curvature \cite{Koi2} and in symmetric spaces of compact type \cite{Koi3};
for spacelike submanifolds in Lorentz-Minkowski space \cite{Izu}; 
for complex submanifolds in complex projective space \cite{Hoi}.

 On the other hand, as a generalization of immersed submanifolds, 
classes of submanifolds with singularities, known as {\it wave fronts} or {\it frontals}, 
have been intensively investigated in recent years
\cite{Honda, HNSUY, HS, HTY, Ishi, KS, SUY-f, SUY, USY}.
Kossowski and Scherfner \cite{KS} obtained a Chern-Lashof type theorem 
for $2$-dimensional wave fronts in $\R^3$, 
which can be seen as an intrinsic generalization of the Chern-Lashof theorem,
see Fact~\ref{fa:ks}.
However, there are examples of frontals that are not wave fronts 
and play an important role in the study of total absolute curvature (cf. Example~\ref{ex:surface-2}).
Therefore, it is natural to ask whether
the Chern-Lashof theorem holds for frontals 
rather than wave fronts, and whether such a theorem can be extended to arbitrary dimensions and codimensions.
The relationship between minimal total absolute curvature and convexity of the image also remains unclear.

In this paper, 
we prove the following Chern-Lashof type theorem for admissible frontals.
For the definitions of admissible frontals and their total absolute curvature, 
see Definitions~\ref{def:admissible} and \ref{def:TAC}.

\medskip

\begin{introtheorem}
\label{thm:intro1}
Let $M^n$ be a compact oriented manifold of dimension $n$,
and let $f:M^n \to \R^{n+r}$ be a co-orientable admissible frontal.

\begin{itemize}
\item[$(1)$] 
Let $b_i (M^n)$ be the $i$-th Betti number of $M^n$ $(0 \leq i \leq n)$.
Then, the total absolute curvature $\tau (M^n,f)$ of the frontal $f$ satisfies the inequality
$$
\tau(M^n,f) \geq \sum_{i=0}^n b_i (M^n).
$$
\item[$(2)$]  
If the total absolute curvature $\tau(M^n,f)$ is less than $3$, then $M^n$ is homeomorphic to an $n$-dimensional sphere.
\item[$(3)$]  
If the total absolute curvature $\tau(M^n,f)$ is equal to $2$, 
then the image $f(M^n)$ is contained in an $(n+1)$-dimensional affine subspace of $\R^{n+r}$. 
\end{itemize}
\end{introtheorem}

\medskip
  
We remark that the assertion $(3)$ in Theorem~\ref{thm:intro1} does not assert any convexity property.
If we restrict the type of singularities to the first kind (see Definition~\ref{def:sing-kind}), 
we obtain the following relationship between minimal total absolute curvature and convexity of the image.

\medskip

\begin{introtheorem}
\label{thm:intro2}
Let $M^n$ be a compact oriented manifold of dimension $n$,
and let $f:M^n \to \R^{n+r}$ be a co-orientable frontal.
Suppose that the singular set $\Sigma_f$ is not empty and consists of singular points of the first kind.
Then, the total absolute curvature $\tau(M^n,f)$ is equal to $2$ if and only if 
the frontal $f$ satisfies the following three conditions.
\begin{itemize}
\item[\rm{ (a)}] 
The manifold $M^n$ and the singular set $\Sigma_f$ are homeomorphic to an $n$-dimensional sphere and an $(n-1)$-dimensional sphere, respectively,
\item[\rm{ (b)}] 
the image $f(M^n)$ is a closed convex domain contained in an $n$-dimensional affine subspace of $\R^{n+r}$, and
\item[\rm{ (c)}] 
the image $f(\Sigma_f)$ coincides with the boundary of $f(M^n)$.
\end{itemize}
\end{introtheorem}

\medskip

Since the frontals satisfying the conditions of Theorem~\ref{thm:intro2} are not wave fronts, 
the generalization of 
the Chern-Lashof theorem to frontals (Theorem~\ref{thm:intro1}) is essential for proving Theorem~\ref{thm:intro2}.
In addition, it is interesting that the dimension of the affine subspace containing the image drops to $n$
 (cf.\ Remark~\ref{rem:n-subspace}).

This paper is organized as follows.
In Section~\ref{sec:frontal}, we describe some basic materials of frontals and define admissible frontals
and their total absolute curvature. In Section~\ref{sec:proof}, we prove Theorems A and B.
Finally, in Section~\ref{sec:low-dimensional}, we apply our results to the $1$- and $2$-dimensional cases.

\section{Admissible frontals and the total absolute curvature}
\label{sec:frontal}
In this section,  we define admissible frontals (Definition~\ref{def:admissible}) and 
the total absolute curvature for them (Definition~\ref{def:TAC}).

\subsection{Frontals and singularities}\label{sub:frontals}
In this subsection,
we review the definition of frontals in $(n+r)$-dimensional Euclidean space $\R^{n+r}$.
The notion of frontals as generalized submanifolds is introduced in \cite{Ishi}.
The cases of surfaces and hypersurfaces are well investigated, see \cite{SUY-f, USY} for example.

Let $M^n$ be an oriented $n$-manifold.
For a smooth map $f:M^n \to \R^{n+r}$,
a point $p \in M^n$ is called a {\it singular point} of $f$ 
if $f$ is not an immersion at $p$.
Otherwise, $p$ is called a {\it regular point}.
Let $\Sigma_f$ be the singular set of $f$, 
and define $M^n_{\rm{reg}}$ by $M^n_{\rm{reg}} := M^n \setminus \Sigma_f$.
We denote by $\widetilde{Gr}(n,n+r)$ the Grassmannian of oriented $n$-dimensional subspaces of $\R^{n+r}$.
A smooth map $f:M^n \to \R^{n+r}$ such that $M^n_{\rm{reg}}$ is dense in $M^n$ is called a {\it frontal\/} 
 if, for each $ p \in M^n$, there exist an open neighborhood $U$ and
a smooth map $\Pi: U \to \widetilde{Gr}(n,n+r)$ such that
$$
df_q (X) \in \Pi(q) \hspace{6mm} ( q \in U, \,  X \in T_q M^n).
$$
This map $\Pi$ is called the {\it generalized Gauss map} of $f$.
If $\Pi$ can be defined globally on $M^n$, then $f$ is said to be {\it co-orientable}.
If $f$ is an immersion, then $f$ is co-orientable.
This can be verified by setting $\Pi(q) := df_q (T_q M^n)\,\, (q \in M^n)$.
By taking the double cover of $M^n$ if necessary, we may assume that frontals are co-orientable without loss of generality.
Therefore, in this paper, frontals are assumed to be co-orientable.

For the generalized Gauss map $\Pi$, we denote by  $\Pi^{\bot} (p)$ the orthogonal complement of $\Pi(p)$.
For a local coordinate system $(U; u_1,u_2, \cdots , u_n)$ of $M^n$, we define a function $\lambda$ on $U$ by
$$
\lambda(p)  := 
\det (f_1 (p) , f_2 (p) , \cdots , f_n (p), E_1 (p) , \cdots , E_{r}(p) )\quad (p \in U)
$$
where $f_i  :=  \partial f / \partial u_i$ and 
$\{ E_1 , \cdots , E_{r} \}$ is an orthonormal frame of $\Pi^{\bot}$.
The function $\lambda$ is called the {\it signed volume density function}.
Then, a point $p \in M^n$ is a singular point of $f$ if and only if $\lambda(p)$ is equal to $0$.
A singular point $p$ is said to be {\it non-degenerate} if the exterior derivative $d \lambda$ 
does not vanish at $p$.
Non-degeneracy of singular points does not depend on the choices of coordinates of the domain and an orthonormal frame 
$\{ E_1 , \cdots , E_{r} \}$ of $\Pi^{\bot}$.

\medskip
\begin{lemma}\label{le:rank}
If a singular point $p$ is non-degenerate, then the rank of $df_p$ is equal to $n-1$.
\end{lemma}
\medskip

\begin{proof}
Let $(U; u_1,u_2, \cdots , u_n)$ be a local coordinate neighborhood of $p$.
For each $X \in T_p M$, we have
\begin{multline}
d \lambda_p (X) = \sum_{i=1}^{n} \det (f_1 (p) ,  \cdots, (d f_i)_p (X), \cdots, f_n (p), E_1 (p) , \cdots , E_{r}(p) ) \notag \\
+ \sum_{j=1}^{r} \det (f_1 (p) ,  \cdots , f_n (p), E_1 (p) , \cdots ,(d E_j)_p (X) ,\cdots , E_{r}(p) ).
\end{multline}
Since $p$ is a singular point, the rank of $df_p$ is less than $n$.
So, 
$$
\det (f_1 (p) ,  \cdots , f_n (p), E_1 (p) , \cdots ,(d E_j)_p (X) ,\cdots , E_{r}(p) ) =0
$$
holds for $1 \leq j \leq r$.
If the rank of $df_p$ is less than $n-1$, then
$$
\det (f_1 (p) ,  \cdots, (d f_i)_p (X), \cdots, f_n (p), E_1 (p) , \cdots , E_{r}(p) ) =0
$$
holds for $1 \leq i \leq n$.
Thus,  the derivative $d \lambda_p (X)$ vanishes. 
This contradicts the non-degeneracy of the singular point $p$.
Therefore, the rank of $df_p$ is equal to $n-1$.
\end{proof}

The implicit function theorem yields that if singular points are all non-degenerate, then the singular set $\Sigma_f$ is a regular hypersurface of $M^n$.
By Lemma~\ref{le:rank}, the kernel $\ker df_p$ of $df_p$ is a $1$-dimensional subspace of $T_p M^n$ at a non-degenerate singular point $p$.
We call $\ker df_p$ the {\it null space} at $p$.

\medskip
\begin{definition}\label{def:sing-kind}
Suppose that all singular points of the frontal $f: M^n \to \R^{n+r}$ are non-degenerate.
A non-degenerate singular point 
$p$ is said to be of the {\it first kind} 
if the null space $\ker df_p$ is not contained in $T_p \Sigma_f$;
otherwise, $p$ is called the {\it second kind}.
\end{definition}
\medskip

We denote by $C_f\, (\subset \Sigma_f)$ the set of singular points of the second kind.

\medskip
\begin{definition}\label{def:admissible}
If a frontal $f:M^n \to \R^{n+r}$ satisfies the following conditions, $f$ is called an {\it admissible frontal}.
\begin{itemize}
\item[$(1)$] 
All singular points of $f$ are non-degenerate.
\item[$(2)$] 
There exists a regular hypersurface $\mathcal{H}$ in $\Sigma_f$ such that $C_f  \subset \mathcal{H}$.
\end{itemize}
\end{definition}
\medskip

We remark that a frontal
whose singular set consists of singular points of the first kind
is admissible (cf.\ Theorem~\ref{thm:intro2}).
Definition~\ref{def:admissible} is inspired by ``admissible singular points of the second kind'' in \cite{USY}  (cf.\ Remark~\ref{re:admissible} ).
By definition, if $f : M^n \to \R^{n+r}$ is an admissible frontal, then the set $\Sigma_f \setminus C_f$ of singular points of the first kind is dense in the singular set $\Sigma_f$.
This property is used in the definition of total absolute curvature (Definition~\ref{def:TAC}), in Proposition~\ref{prop:bar-f} and in Lemma~\ref{le:bundle-measure}. 
In the case of $n=2$, the condition ($2$) of Definition~\ref{def:admissible} is equivalent to the condition that
the set of singular points of the second kind $C_f$ is a discrete subset of $\Sigma_f$ (cf. \cite{KS}, Fact~\ref{fa:ks}).
Moreover, frontals with at most $A_k$-type singular points are admissible:

\medskip
\begin{example}\label{ex:A3sing}
Suppose that each singular point of the frontal $f: M^n \to \R^{n+1}$ is right-left-equivalent to the {\it  $A_{k+1}$-type singular point }$(1 \leq k \leq n)$.
Here, the  $A_{k+1}$-type singular point is the map-germ of the map from $\R^n$ to $\R^{n+1}$ defined by
$$
(t,x_2, \cdots , x_n) \mapsto 
\left(
(k+1) \sum_{j=2}^k (j-1)\,t^j x_j , -(k+2) t^{k+1}- \sum_{j=2}^k j \,t^{j-1} \, x_j, x_2, \cdots , x_n 
\right)
$$
at the origin of $\R^n$.
For an $A_{k+1}$-type singular point $p \in M^n$ and the neighborhood $U$ of $p$, 
we give the following notation:
$$
\lambda^{(0)} = \lambda, \quad \lambda^{(1)} = \eta\lambda, \quad \lambda^{(i)} = d \lambda^{(i-1)} (\eta) \quad (2 \leq i \leq k),
$$
where $\lambda$ is the signed volume density function of $f$ and $\eta$ is a non-vanishing vector field on $U$ such that $\eta(q) \in \ker df_q$ for each singular point $q \in \Sigma_f \cap U$.
Then, 
$$
\lambda(p) = \lambda'(p) = \cdots = \lambda^{(k-1)} (p) = 0, \quad \lambda^{(k)}  (p) \neq 0
$$
hold and the Jacobian matrix of 
the smooth map $ (\lambda , \lambda' , \cdots , \lambda^{(k-1)} ): U \to \R^{k}$
 is of rank $k$ at $p$
(\cite[ Corollary $2.5$]{SUY}).
In this case, all singular points of $f$ are non-degenerate.
We define the map $\psi:U \to \R^2$  by $\psi = (\lambda, \lambda') $.
Since a singular point $p \in U$ is of the second kind if and only if $\lambda(p) = \lambda'(p) = 0$, 
the set $C_f \cap U$ is given by $\psi^{-1} ((0,0))$.
Now, the Jacobian matrix of $\psi$ is of rank $2$ at each $p \in \psi^{-1} ((0,0))$.
Hence, $(0,0) \in \R^2$ is a regular value of $\psi$.
Therefore, $f$ is an admissible frontal by setting $\mathcal{H} = C_f$.
\end{example}
\medskip

For example, in the case of the map $f:\R^3 \to \R^4$ of the $A_{4}$-type singular point
$$
f(t,x,y) = ( 4 t^5 + t^2 x + 2 t^3 y , 5 t^4 + 2 t x + 3 t^2 y , x, y ),
$$
we obtain $\Sigma_f = \{ (t, -3ty- 10 t^3, y) \in \R^3 \mid (t,y) \in \R^2 \}$ and
$$
C_f = \{ (t, -3ty- 10 t^3, y) \in \Sigma_f \mid y = -10t^2 \} = \{ (t, 20t^3, -10t^2) \in \Sigma_f \mid t \in \R \}.
$$
Thus, $f$ is an admissible frontal (see Figure~\ref{fig:A3}).

\begin{figure}[htb]
\centering
 \begin{tabular}{c}
\resizebox{4cm}{!}{\includegraphics{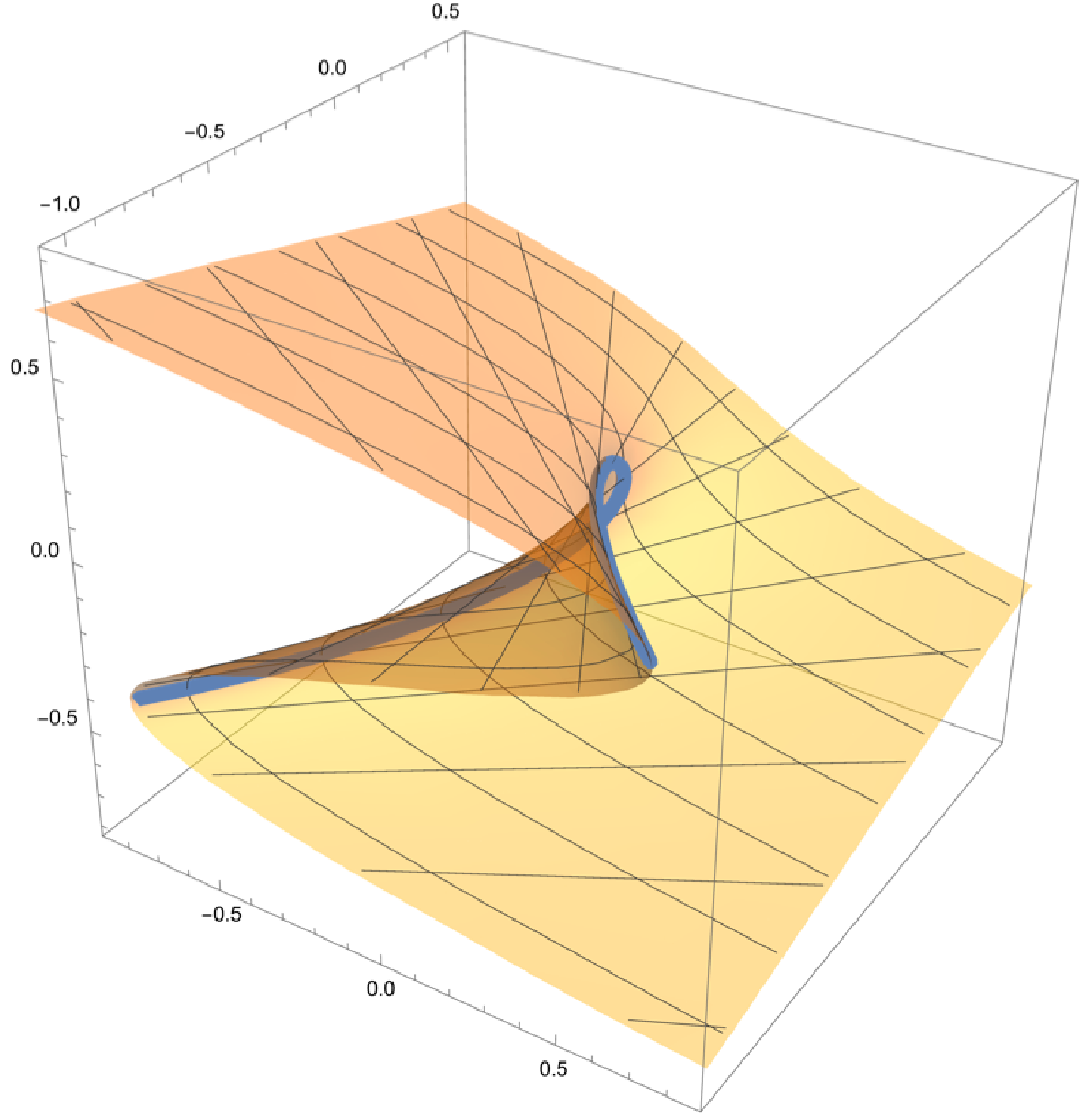}}
\end{tabular}  
\caption{The singular set $\Sigma_f$ (yellow surface) and the set of singular points of the second kind $C_f$ (blue curve) of the map-germ $f:\R^3 \to\R^4$ of the $A_{4}$-type singular point. 
Since $C_f$ is a regular curve in $\Sigma_f$, $f$ is an admissible frontal 
(cf.  Definition~\ref{def:admissible}, Example~\ref{ex:A3sing}).}
\label{fig:A3}
\end{figure}

\subsection{The total absolute curvature}
In this subsection, we review the notion of the Lipschitz-Killing curvature
and define the total absolute curvature for admissible frontals (Definition~\ref{def:TAC}).
 Let $f:M^n \to \R^{n+r}$ be a co-orientable frontal.
We define the unit normal bundle $B$ of $f$ by 
$$
B  :=  \bigcup_{p \in M^n} B_p \qquad \left( B_p := \{ \vect{v} \in \Pi^{\bot}(p) \mid \| \vect{v} \| = 1 \} \right).
$$
Here, $\| \vect{w} \| := \sqrt{\vect{w} \cdot \vect{w}}$ is 
the norm of a vector $\vect{w} \in \R^{n+r}$,
and $\vect{w} \cdot \vect{v}$ is the canonical Euclidean inner product of $\vect{w},\vect{v} \in \R^{n+r}$.
A smooth section of $B$ is called a {\it unit normal vector field} of $f$.
Let $N$ be a unit normal vector field of $f$.
On the regular set $M^n_{\rm{reg}}$, the following Weingarten formula holds:
\begin{equation}\label{eq:Weingarten}
 D_X N 
=
 -df(A_{N} X) + D^{\bot}_X N.
\end{equation}
Here, $X$ is a smooth vector field on $M^n$,
$D$ is the canonical connection of $\R^{n+r}$, $D^{\bot}$ is the normal connection, and  $A_{N}$ is the shape operator with respect to $N$.
 Define 
$B_{\rm{reg}}$
by
$$
B_{\rm{reg}} := \bigcup_{p \in M^n_{\rm{reg}}} B_p.
$$

For each $(p,\xi) \in B_{\rm{reg}}$, 
we define the {\it Lipschitz-Killing curvature} $G(p,\xi)$ by
$$
G(p,\xi)  :=  \det A_{\xi}. 
$$

We let $(U;u_1, \cdots , u_n)$ be a coordinate system of $M^n$.
The two $n$-forms 
$$
d\hat{V}  := \lambda\, du_1 \wedge \cdots \wedge du_n, 
\quad dV  := |\lambda| \, du_1  \wedge \cdots \wedge du_n 
$$
on $U$ are called the {\it signed volume element} and the ({\it unsigned}\/) {\it volume element} of $f$, respectively.
The signed volume element of $f$ does not depend on the choice of coordinate systems of $M^n$ or the choice of an orthonormal frame
compatible with the orientation of $\Pi^{\bot}$.
The (unsigned) volume element of $f$ also does not depend on the choice of coordinate systems
compatible with the orientation of $M^n$ and choice of an orthonormal frame
compatible with the orientation of $\Pi^{\bot}$.
Let $d \sigma$ be the volume element of each fiber of $B$.
The {\it signed volume element} and the ({\it unsigned}\/) {\it volume element} of $B$ are defined by
$$
d \hat{\mu}_B  :=  d \hat{V} \wedge d\sigma, 
\quad d \mu_B  :=  d V \wedge d\sigma,
$$
respectively.
Let
\[
S^{n+r-1}:=\{\vect v\in\R^{n+r}\mid \|\vect v\|=1\}
\]
be the unit sphere in $\R^{n+r}$.
We call the smooth map
$$\nu : B \to S^{n+r-1} \,; (p,\xi) \mapsto \xi
$$
 the {\it canonical Gauss map} of $f$. 
The following proposition is important for the definition of the total absolute curvature (Definition~\ref{def:TAC}).

\medskip

\begin{proposition}\label{prop:pullback}
 Let $f:M^n \to \R^{n+r}$ be a co-orientable frontal,
and let
$d \mu_{S^{n+r-1}}$ be the volume element of $S^{n+r-1}$.
On $B_{\rm{reg}}$, the pullback of $d \mu_{S^{n+r-1}}$ by $\nu$ can be written as
$$
\nu^* d \mu_{S^{n+r-1}} = (-1)^n G(p,\xi) \, d \hat{\mu}_{B}.
$$
\end{proposition}

\medskip

\begin{proof}
We fix $(p,\xi) \in B_{\rm{reg}}$.
Let $(U; u_1, \cdots , u_n)$ be a local coordinate system of $M^n$, 
$(V; \theta_1, \cdots , \theta_{r-1} )$ be a local coordinate system of a unit $(r-1)$-dimensional sphere $S^{r-1}$.
Then, the following holds
\begin{multline}
\nu^* d \mu_{S^{n+r-1}} (p,\xi) \notag \\
= 
\det \left(
\nu_{u_1}, \cdots, \nu_{u_n},
\nu_{\theta_1}, \cdots ,
\nu_{\theta_{r-1}},\nu
\right) (p,\xi)\,
du_1 \wedge \cdots \wedge du_n \wedge 
d\theta_1 \wedge \cdots \wedge d\theta_{r-1},
\end{multline}
where $\nu_{u_i}  :=  \partial \nu / \partial u_i \, (1 \leq i \leq n)$ and 
$\nu_{\theta_j}  :=  \partial \nu / \partial \theta_j \, (1 \leq j \leq r-1)$.
By the Weingarten formula \eqref{eq:Weingarten}, we have
$$
\nu_{u_i} (p,\xi)
= -df \left(
A_{\xi} \, \partial_{u_i}
\right) (p)
+ D^{\bot}_{\partial_{u_i}} \nu \, (p,\xi),
$$
where $\partial_{u_i}  :=  \partial / \partial u_i$.
Thus, 
\begin{align*}
\det &\left(
\nu_{u_1}, \cdots, \nu_{u_n},
\nu_{\theta_1}, \cdots ,
\nu_{\theta_{r-1}},\nu 
\right) (p,\xi) \\
&= (-1)^n \det \left(
df \left( A_{\xi} \, \partial_{u_1} \right) , \cdots , df \left( A_{\xi} \, \partial_{u_n} \right)  , 
\nu_{\theta_1}, \cdots ,
\nu_{\theta_{r-1}},\nu 
\right) (p,\xi) \\
&= (-1)^n \det A_{\xi} \det \left(
f_1,\cdots , f_n, 
\nu_{\theta_1}, \cdots ,
\nu_{\theta_{r-1}},\nu 
\right) (p,\xi)
\end{align*}
holds. Therefore, 
\begin{align*}
\det A_{\xi} &\det \left(
f_1,\cdots , f_n, 
\nu_{\theta_1}, \cdots ,
\nu_{\theta_{r-1}},\nu 
\right)(p,\xi) du_1 \wedge \cdots \wedge du_n \wedge 
d\theta_1 \wedge \cdots \wedge d\theta_{r-1}
\\
&=  G(p,\xi) \, \lambda (p) \,
du_1 \wedge \cdots \wedge du_n \wedge d \sigma \\
&=  G(p,\xi)\, d \hat{\mu}_{B}
\end{align*}
implies the desired result.
\end{proof}
By Proposition~\ref{prop:pullback}, the differential form  $(-1)^n G\, d \hat{\mu}_{B}$ can be smoothly extended to whole on $B$.
Hence, we obtain the following corollary.

\medskip

\begin{corollary}\label{cor:curvature-form}
 Let $f:M^n \to \R^{n+r}$ be a co-orientable frontal,
and let $B$ be the unit normal bundle of $f$. 
Then, the $(n+r-1)$-form $|G|\,d \mu_B$, 
defined on $B_{\rm reg}$, 
extends continuously to $B$.
\end{corollary}

\medskip

From now on, we assume that 
$f:M^n\to \R^{n+r}$ is a co-orientable admissible frontal.
We denote by $\bar f$ the restriction of $f$ 
to the singular set $\Sigma_f$. 
For $p\in\Sigma_f$, 
Lemma~\ref{le:rank} implies that
$df_p(T_pM^n)$ is an $(n-1)$-dimensional subspace of $\R^{n+r}$.
We define
$$
\bar{\Pi}:\Sigma_f\to \widetilde{Gr}(n-1,n+r),
\qquad
\bar{\Pi}(p):=df_p(T_pM^n).
$$

\medskip

\begin{proposition}\label{prop:bar-f}
Let $f:M^n\to\R^{n+r}$ be a co-orientable admissible frontal.
Then $\bar f:=f|_{\Sigma_f}:\Sigma_f\to\R^{n+r}$ is a co-orientable
frontal with generalized Gauss map $\bar\Pi$.
\end{proposition}

\medskip

\begin{proof}
First, we show that 
$\bar\Pi$ is a generalized Gauss map
along $\bar f$.
In fact, for each $p\in\Sigma_f$, we have
$$
d\bar f_p(T_p\Sigma_f)\subset df_p(T_pM^n)=\bar\Pi(p).
$$
Moreover, 
it follows from Lemma~\ref{le:rank} that
$df$ has constant rank $n-1$ on $\Sigma_f$.
Hence, for each $p\in\Sigma_f$, 
there exist a neighborhood $W$ of $p$ in $M^n$
and smooth vector fields $X_1,\ldots,X_{n-1}$ on $W$
such that
$
df_q((X_1)_q),\ldots,df_q((X_{n-1})_q)
$
form a basis of $df_q(T_qM^n)$ for each $q\in\Sigma_f \cap W$. 
Namely, we have
$$
\bar\Pi(q)
=
\operatorname{span}
\{df_q((X_1)_q),\ldots,df_q((X_{n-1})_q)\}
$$
for each $q\in\Sigma_f \cap W$.
Thus $\bar\Pi$ is smooth.

It suffices to show that 
the regular set of $\bar f$ is dense in $\Sigma_f$.
We fix a singular point $p\in\Sigma_f\setminus C_f$
of the first kind arbitrarily.
Since the null space $\ker df_p$ is not contained in $T_p\Sigma_f$,
$$
d\bar f_p:T_p\Sigma_f\to\R^{n+r}
$$
is injective.
Namely, $\bar f$ is an immersion at 
a singular point of the first kind.

Since $f$ is admissible, 
the set $\Sigma_f\setminus C_f$ of singular points of the first kind 
is dense in $\Sigma_f$. 
Thus the regular set of $\bar f$ is dense in
$\Sigma_f$, and \(\bar f\) is a frontal. 
\end{proof}

\medskip

\begin{remark}\label{re:bar-f}
If all singular points of $f$ are of the first kind, 
namely, if $C_f$ is empty,
then $\bar f : \Sigma_f \to \R^{n+r}$
is an immersion.
We will use this fact in the proof of Theorem~\ref{thm:intro2}
in Section~\ref{sec:proof}.
\end{remark}

\medskip

Now, we let $\bar{B},\bar{A},\bar{\nu}$ and $\bar{G}$ be 
the bundle of unit normal vectors, the shape operator, the canonical Gauss map and the Lipschitz-Killing curvature of $\bar{f}$, respectively.

By Corollary~\ref{cor:curvature-form} and Proposition~\ref{prop:bar-f},
both $|G|\,d \mu_B$ and $|\bar{G}|\,d \mu_{\bar{B}}$
extend continuously to $B$ and $\bar{B}$, respectively.
Hence, 
we may define the total absolute curvature of $f$ as follows:

\medskip
\begin{definition}\label{def:TAC}
Let $M^n$ be a compact orientable $n$-dimensional manifold.
For a co-orientable admissible frontal $f : M^n \to \R^{n+r}$, the {\it total absolute curvature} $\tau(M^n,f)$ is defined by 
$$
\tau (M^n,f)  :=  
\frac{1}{\operatorname{vol} (S^{n+r-1})} \int_B |G|\, d \mu_B 
+ \frac{1}{\operatorname{vol} (S^{n+r-1})} \int_{\bar{B}} |\bar{G}|\, d \mu_{\bar{B}}.
$$
\end{definition}
\medskip

If we consider only the total absolute curvature on a regular set of $f$, it has no non-trivial lower bound, 
as shown in Example~\ref{ex:surface-1} below.

\medskip

\begin{example}\label{ex:surface-1}
               
		For a positive number $k > 0$, we define $f_k : S^n \to \R^{n+1}$ by 
		$$
		f_k (x_1, \cdots, x_n , x_{n+1}) 
		 :=(x_1, \cdots, x_n , k x_{n+1}^3 ) \quad ((x_1, \cdots, x_n , x_{n+1}) \in S^n \subset \R^{n+1}).
		$$
		Then, $f_k$ is an admissible frontal 
		 (see Figure~\ref{fig:mgt}).
		The total absolute curvature 
		of $f_k$ on the regular set satisfies
		$$
		\frac{1}{\operatorname{vol} (S^{n})} \int_B |G| \, d \mu_B 
		=
		\frac{6 \operatorname{vol} (S^{n-1})}{\operatorname{vol} (S^{n})} 
		\int_{-\pi/2}^{\pi/2} \frac{k |\cos (t)|}{(1 + (k \sin t)^2 / 4)^{\frac{n+1}{2}}} \, dt 
		<
		\frac{12k \operatorname{vol} (S^{n-1}) }{\operatorname{vol} (S^{n})}.
		$$
		
		Therefore, for any $\varepsilon>0$, by choosing
		\[
		0<k<\frac{\varepsilon \operatorname{vol}(S^n)}
		{12\operatorname{vol}(S^{n-1})},
		\]
		the contribution of the regular part of $f_k$ to the total absolute curvature is less than $\varepsilon$.
		
\end{example}

\begin{figure}[htb]
\centering
 \begin{tabular}{c}
\resizebox{4cm}{!}{\includegraphics{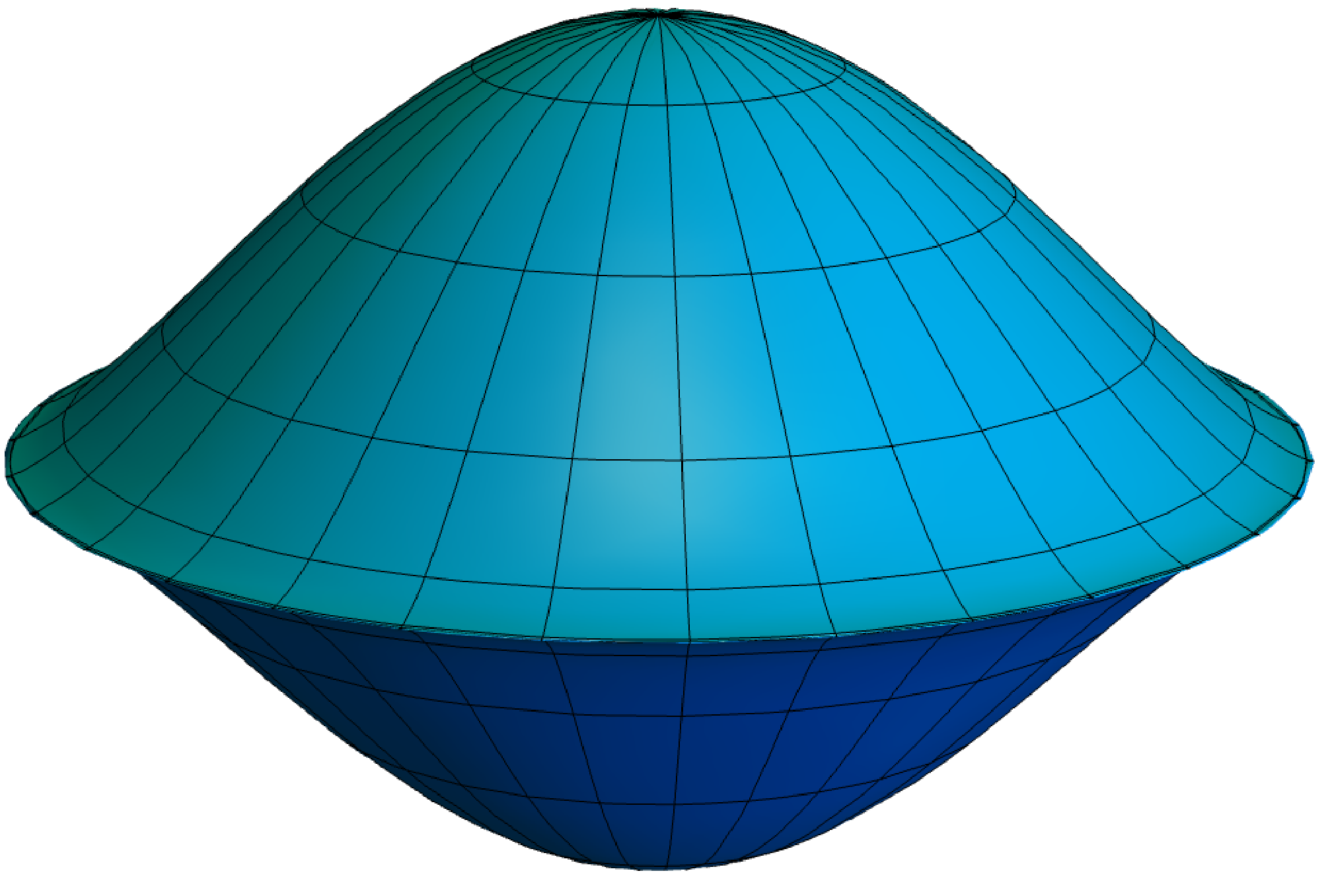}}
\end{tabular}  
\caption{The image $f_k (S^2)$ when $n=2$ and $k=\frac{2}{3}$.}
\label{fig:mgt}
\end{figure}

\section{Proof of Theorems A and B}
\label{sec:proof}

In this section, we prove Theorems~\ref{thm:intro1} and \ref{thm:intro2} by applying Morse theory to height functions.

Throughout this section,
we fix a co-orientable admissible frontal $f:M^n \to \R^{n+r}$.

\subsection{Proof of Theorem~\ref{thm:intro1}}

For a unit vector $\vect{w} \in S^{n+r-1}$,
define a function $h_{\vect{w}}$ on $M^n$ as 
$$
h_{\vect{w}} (p) := f(p) \cdot \vect{w} \quad (p \in M^n).
$$
We call $h_{\vect{w}}$ the {\it height function} with respect to $\vect{w}$.
If the derivative $(d h_{\vect{w}})_p$ vanishes at $p \in M^n$,
then $p$ is called a {\it critical point} of $h_{\vect{w}}$.
For a point $p \in M^n$, it holds that
\begin{equation}\label{eq:crit-pt}
\text{
$p$ is a critical point of $h_{\vect{w}}$ 
if and only if $\vect{w}$ is perpendicular to  $df_p (T_p M^n)$.}
\end{equation}
Furthermore, 
if the Hessian determinant  
of $h_{\vect{w}}$ is non-zero at a critical point $p$, 
then $p$ is called a {\it Morse critical point}.
At a regular point $p \in M^n_{\rm{reg}}$,
the following holds (see \cite{CL2}):
\begin{equation}\label{eq:MCP-regular}
\begin{minipage}{0.8\linewidth}
{\it 

If $p \in M^n_{\rm{reg}}$ is a critical point of $h_{\vect{w}}$,
then $p$ is a Morse critical point if and only if 
 $(p,\vect{w})$ is a regular point of the canonical Gauss map $\nu$.
}\end{minipage}
\end{equation}
A similar argument holds true for a singular point of the first kind as follows.

\medskip
\begin{lemma}\label{le:first-morse}
Let $p \in \Sigma_f \setminus C_f$ be a singular point of the first kind,
 and let $\vect{w}\in S^{n+r-1}$.
Suppose that $p$ is a critical point of the height function $h_{\vect{w}}$.
Then $p$ is a Morse critical point of $h_{\vect{w}}$ 
if and only if 
 $(p,\vect{w})$ is a regular point of the canonical Gauss map 
$\bar{\nu}$ of $\bar{f}$ and is not perpendicular to $\Pi(p)$.
\end{lemma}
\medskip

\begin{proof}

Since $p$ is a singular point of the first kind, we may choose local coordinates
$(u_1,\ldots,u_n)$ centered at $p$ such that
\[
\Sigma_f\cap U=\{u_n=0\},\qquad
\ker df_q=\operatorname{span}\{\partial_{u_n}\}
\quad(q\in\Sigma_f\cap U).
\]
In particular,
$
f_n(u_1,\ldots,u_{n-1},0)=\vect 0 .
$
Then there exists a smooth map $\vect{v} : U \to \R^{n+r}$ such that $f_n = u_n\, \vect{v}$.
Then we have
$$
f_{in} (p) = \vect{0} \hspace{6mm} (1 \leq i \leq n-1), \qquad
f_{nn} (p) = \vect{v} (p),
$$
where  $f_{ij}  := \partial^2 f / \partial u_i \partial u_j \,\,(1 \leq i,j \leq n) $.
 The vector fields $\partial_{u_i}$ $(1\leq i\leq n-1)$ are tangent to $\Sigma_f$.
Therefore, we obtain 
$$
f_{ij} (p) \cdot \vect{w} = df \left( \bar{A}_{\vect{w}} \, \partial_{u_i} \right) \, (p) \cdot f_j (p)\quad (1 \leq i,j \leq n-1).
$$

Consequently, we have 
$$
\det [f_{ij}(p) \cdot \vect{w}]_{(1 \leq i,j \leq n)} 
= (\vect{v} (p) \cdot \vect{w}) \det \left[df \left( \bar{A}_{\vect{w}} \, \partial_{u_i} \right) (p) \cdot f_j (p) \right]_{(1 \leq i,j \leq n-1)}.
$$
As a result, $p$ is a Morse critical point if and only if
both $\vect{v} (p) \cdot \vect{w}$ and $\det [df(\bar{A}_{\vect{w}}\partial /\partial u_i) \cdot f_j]_{(1 \leq i,j \leq n-1)}$ are non-zero.
Since $p$ is a critical point of $h_{\vect{w}}$, 
the unit vector $\vect{w}$ is perpendicular to 
 $df_p (T_p \Sigma_f)$.
Hence, $\vect{v} (p) \cdot \vect{w}$ is non-zero if and only if $\vect{w}$ is not perpendicular to $\Pi(p)$.
The determinant $\det [df(\bar{A}_{\vect{w}} \, \partial_{u_i}) \cdot f_j]_{(1 \leq i,j \leq n-1)}$ is non-zero if and only if 
$\bar{A}_{\vect{w}}$ is non-degenerate, that is, 
 $(p,\vect{w})$ is a regular point 
of the canonical Gauss map $\bar{\nu}$ of $\bar{f}$.
Thus, we have the desired result.
\end{proof}

As we will see later, the proof of Theorem~\ref{thm:intro1} 
does not require any conditions 
for a singular point of the second kind to be a Morse critical point
(cf.\ Lemma~\ref{le:second-zero}, Lemma~\ref{le:measure}).

Define subsets $Y_1$ and $Y_2$ of $S^{n+r-1}$ by
$$
Y_1 := \nu \left( \bigcup_{p \in \Sigma_f} B_p \right), \quad 
Y_2 := \bar{\nu} \left( \bigcup_{p \in C_f} \bar{B}_p \right).
$$
We remark that $Y_1$ (resp. $Y_2$) is the set of unit normal vectors of $f$ on $\Sigma_f$ (resp. $\bar{f}$ on $C_f$).
The following three lemmas 
 (Lemmas~\ref{le:bundle-measure}, \ref{le:second-zero} and \ref{le:measure}) are used in the proof of Theorem~\ref{thm:intro1}.

\medskip
\begin{lemma}\label{le:bundle-measure}
Both $Y_1$ and $Y_2$ have measure zero in $S^{n+r-1}$.
\end{lemma}
\medskip

\begin{proof}
Since the bundle $ \bigcup_{p \in \Sigma_f} B_p$ is locally diffeomorphic to $\Sigma_f \times S^{r-1}$, 
the bundle $\bigcup_{p \in \Sigma_f} B_p$ is an $(n+r-2)$-manifold.
Therefore, $Y_1$ has measure zero.
Since $f$ is an admissible frontal, there exists a regular hypersurface $\mathcal{H}$ such that $C_f  \subset \mathcal{H}$.
Thus, it suffices to show that $\bar{\nu} ( \bigcup_{p \in \mathcal{H}} \bar{B}_p )$ has measure zero.
Since $\bigcup_{p \in \mathcal{H}} \bar{B}_p$ is locally diffeomorphic to $\mathcal{H} \times S^{r}$, 
the bundle $ \bigcup_{p \in \mathcal{H}} \bar{B}_p$ is an $(n+r-2)$-manifold.
Thus, $\bar{\nu} ( \bigcup_{p \in \mathcal{H}} \bar{B}_p )$ has measure zero in $S^{n+r-1}$.
Therefore, $Y_2$ has measure zero in $S^{n+r-1}$.
\end{proof}

\begin{lemma}\label{le:second-zero}
If a unit vector $\vect{w}$ does not belong to $Y_2$, 
then there is no critical point of the height function  $h_{\vect{w}}$ on 
 the set $C_f$ of the singular points of the second kind.
\end{lemma}

\begin{proof}
We prove the contraposition.
If there exists a critical point $p \in C_f$ of $h_{\vect{w}}$, then $\vect{w}$ is perpendicular to  $\bar{\Pi} (p) =df_p (T_p M^n)$
(cf.\ \eqref{eq:crit-pt}).
In other words, $\vect{w}$ is a unit normal vector of $\bar{f}$ at $p$. Thus, $\vect{w}$ belongs to $Y_2$.
\end{proof}

We define the set $Q \subset S^{n+r-1}$ by
\begin{equation}\label{eq:setQ}
Q  := \{ \vect{w} \in S^{n+r-1} \mid h_{\vect{w}} \mbox{ is not a Morse function} \} \cup Y_2.	
\end{equation}

\begin{lemma}\label{le:measure}
The set $Q$ has measure zero in $S^{n+r-1}$.
\end{lemma}

\begin{proof}
For $\vect{w} \in S^{n+r-1}$, we characterize the condition
that the height function $h_{\vect{w}}$ is not a Morse function.
If $\vect{w} \notin Y_2$,
Lemma~\ref{le:second-zero} yields that there is no critical point of $h_{\vect{w}}$ on $C_f$. 
Therefore, we only need to characterize the condition that a regular point or a singular point of the first kind is a Morse critical point.
 By \eqref{eq:MCP-regular},
a point $p \in M^n_{\rm{reg}}$ is a Morse critical point if and only if 
 $(p,\vect{w})$ is a regular point of the canonical Gauss map $\nu$.
By Lemma~\ref{le:first-morse}, 
a singular point of the first kind $p \in \Sigma_f \setminus C_f$ 
is a Morse critical point if and only if 
 $(p,\vect{w})$ is a regular point 
of the canonical Gauss map $\bar{\nu}$ and is not perpendicular to $\Pi(p)$, 
that is, not a unit normal vector of $f$ at $p$.
 Due to the above, for $\vect w\notin Y_2$,
the height function $h_{\vect{w}}$ is not a Morse function if and only if $\vect{w}$ is a critical value of $\nu$ or $\bar{\nu}$, 
or belongs to $Y_1$.
Hence, we obtain
$$
Q = \{ \vect{w} \in S^{n+r-1} \mid \vect{w} \mbox{ is a critical value of } \nu \mbox{ or } \bar{\nu} \} 
\cup Y_1 \cup Y_2.
$$ 
By Sard's theorem, the first term of the right-hand side has measure zero in $S^{n+r-1}$. 
By Lemma~\ref{le:bundle-measure}, both the second term and the third term have measure zero in $S^{n+r-1}$.
\end{proof}

For a unit vector $\vect{w} \in S^{n+r-1} \setminus Q$, the height function $h_{\vect{w}}$ is a Morse function.
Hence, there are finitely many critical points of $h_{\vect{w}}$ on $M^n$.
For a subset $L  \subset M^n$, 
we denote by $\# \operatorname{crit} (h_{\vect{w}},L)$ the number of critical points of $h_{\vect{w}}$ on $L$.
We prove Theorem~\ref{thm:intro1}-(1).

\begin{proof}[Proof of Theorem~\ref{thm:intro1}-(1)]

Since $Y_2$ is a subset of $Q$, 
Lemma~\ref{le:second-zero} yields that 
$$
\# \operatorname{crit} (h_{\vect{w}},\Sigma_f) = \# \operatorname{crit} (h_{\vect{w}},\Sigma_f \setminus C_f)
$$
 for each $\vect{w} \in S^{n+r-1} \setminus Q$.
By the Morse inequality, we have
$$
\sum_{q = 0}^n b_q (M^n)  \leq \# \operatorname{crit} (h_{\vect{w}},M^n) = 
\# \operatorname{crit} (h_{\vect{w}},M^n_{\rm{reg}}) +  \# \operatorname{crit} (h_{\vect{w}},\Sigma_f \setminus C_f)
$$
for each $\vect{w} \in S^{n+r-1} \setminus Q$.
By integrating both sides over $S^{n+r-1} \setminus Q$, we obtain
\begin{align*}
\left(\sum_{q = 0}^n b_q (M^n) \right) \, \operatorname{vol} (S^{n+r-1}) &\leq 
 \int_{S^{n+r-1} \setminus Q} \# \operatorname{crit} (h_{\vect{w}},M^n_{\rm{reg}})\, d\mu_{S^{n+r-1}} \\
&\quad + \int_{S^{n+r-1} \setminus Q} \# \operatorname{crit} (h_{\vect{w}},\Sigma_f \setminus C_f)\,d\mu_{S^{n+r-1}}.
\end{align*}
Here, we used Lemma~\ref{le:measure} which implies that
 $\int_{S^{n+r-1} \setminus Q}  d\mu_{S^{n+r-1}} $ is equal to the volume of $S^{n+r-1}$.
Moreover, the integral of the number of critical points of the height function on $S^{n+r-1}$ is equal to 
the integral of the absolute value of the Lipschitz-Killing curvature on the unit normal bundle (see the proof of Theorem $1$ in \cite{CL1}).
Therefore, 
$$
\left(\sum_{q = 0}^n b_q (M^n) \right)\, \operatorname{vol} (S^{n+r-1}) \leq \int_{B } |G|\, d \mu_{B} +  \int_{\bar{B}} |\bar{G}\,| d \mu_{\bar{B}}
= \operatorname{vol} (S^{n+r-1}) \tau(M^n,f)
$$
holds, which implies the assertion.
\end{proof}

Now, we prove Theorem~\ref{thm:intro1}-(2).

\begin{proof}[Proof of Theorem~\ref{thm:intro1}-(2)]

By Theorem~\ref{thm:intro1}-(1), we have
\[
\sum_{i=0}^n b_i(M^n) \leq \tau(M^n,f)<3.
\]

If $M^n$ had at least two connected components, 
the compactness and the orientability of $M^n$
yield $b_0(M^n)+b_n(M^n)\geq4$, which is
impossible. Hence $M^n$ is connected.

As shown in the proof of Theorem~\ref{thm:intro1}-(1),
the total absolute curvature $\tau(M^n,f)$ can be written as
\begin{equation}\label{eq:ave-tau}
\displaystyle \tau (M^n,f) 
= \frac{1}{\operatorname{vol} (S^{n+r-1})} \int_{S^{n+r-1}\setminus Q} \# \operatorname{crit} (h_{\vect{w}},M^n) d \mu_{S^{n+r-1}}.
\end{equation}
Since $M^n$ is compact, 
 $\# \operatorname{crit} (h_{\vect{w}},M^n)\geq 2$ for each $\vect{w} \in S^{n+r-1} \setminus Q $.
 If 
there were no $\vect{w} \in S^{n+r-1} \setminus Q$
such that $\# \operatorname{crit} (h_{\vect{w}},M^n) = 2$,
then $\tau(M^n,f) \geq 3$, contradicting the assumption $\tau(M^n,f)<3$.
 Hence, there exists $\vect w\in S^{n+r-1}\setminus Q$ such that
$h_{\vect w}$ is a Morse function with exactly two critical points.
Consequently, by Reeb's theorem, 
$M^n$ is homeomorphic to an $n$-dimensional sphere.
\end{proof}

The proof of Theorem~\ref{thm:intro1}-(3) is given 
in Appendix~\ref{appendix:proof-theorem-a3}.

\subsection{Proof of Theorem~\ref{thm:intro2}}
The following lemma is used in the proof of Theorem~\ref{thm:intro2}.

\medskip
\begin{lemma}\label{le:convex-domain}
Let $M^n$ be an $n$-dimensional orientable compact manifold,
and let $f:M^n \to \R^{n+r}$ be a co-orientable frontal whose singular points are of the first kind.
If there exists a closed convex domain $\Omega$ in an $n$-dimensional affine subspace $\mathcal{P}^n$ of $\R^{n+r}$ such that 
$
f(M^n) \subset \Omega
$
 and $f(\Sigma_f) $ coincides with the boundary of $\Omega$,
then $f(M^n)$ coincides with $\Omega$.
\end{lemma}
\medskip

\begin{proof}
We prove that $\Omega \cap f(M^n)^c$ is empty.
In order to show this, 
we assume that  $\Omega \cap f(M^n)^c$ is not empty, 
leading to a contradiction.
We fix $x \in \Omega \cap f(M^n)^c$ arbitrarily.
We define a function $\Phi_x$ on $M^n$ as $\Phi_x (p) = d(x , f(p))$ $(p \in M^n)$, 
where $d$ is the distance function of $\mathcal{P}^n$ induced from that of $\R^{n+r}$.
Then, there exists a minimum point $p_m \in M^n$ of the function $\Phi_x$.
Since $p_m$ is a critical point of $\Phi_x$,
the minimum point $p_m$ is a singular point of the first kind of $f$, 
and $\vect{y} =x - f(p_m)$ is perpendicular to $df_{p_m} (T_{p_m} M^n)$.
We note that $\vect{y} \neq \vect{0}$ holds by our assumption.
Since $p_m$ is a singular point of the first kind,
we can take the local coordinate neighborhood $(U; u_1 , \cdots, u_n )$ of $p_m$ such that 
$$
U\cap \Sigma_f = (u_1, \cdots , u_{n-1}, 0), \quad p_m = (0, \cdots , 0), \quad f_n ( u_1, \cdots , u_{n-1}, 0) = \vect{0}.
$$
For a positive number $\epsilon > 0$, we set an open interval $I \subset U$ by 
$$
I := \{ (0,\cdots,0,u_n) \in U \mid u_n \in (-\epsilon, \epsilon) \}.
$$
We denote by $\phi := \Phi_x |_I$ the restriction of $\Phi_x$ to $I$.
Since $p_m$ is non-degenerate, $f_{nn} (p_m) \neq \vect{0}$.
Therefore, we have $d^2 \phi / d u_n^2 (0) \neq 0$.
Since $\phi \, (0)$ is a minimum, $d^2 \phi / d u_n^2 (0) > 0$ holds.
In other words,  we have
\begin{equation}\label{eq:min1}
f_{nn} (p_m) \cdot \vect{y} < 0.
\end{equation}
On the other hand, since $\Omega$ is convex, $\Omega$  lies on one side of the supporting hyperplane $H_{p_m} := \{ f(p_m)+ df_{p_m} (X) \in \mathcal{P}^n \mid X \in T_{p_m} M^n \}$ (cf.\ Figure~\ref{fig:convex-domain}).
We set a unit vector $\xi$ by $\xi :=\vect{y} / \|\vect{y} \|$,
and denote by $h_{\xi}$ the height function with respect to $\xi$.
Then the restriction $h_{\xi} |_I$ has also a minimum at $u_n=0$.
Therefore, 
$$
f_{nn} (p_m) \cdot \xi = f_{nn} (p_m) \cdot \frac{\vect{y}}{\|\vect{y} \|}  > 0
$$
holds.
This contradicts \eqref{eq:min1}.
Hence,  $\Omega \cap f(M^n)^c$ is empty.
As a result, $f(M^n)$ coincides with $\Omega$.
\end{proof}

\begin{figure}[htb]
\centering
 \begin{tabular}{c}
\resizebox{6cm}{!}{\includegraphics{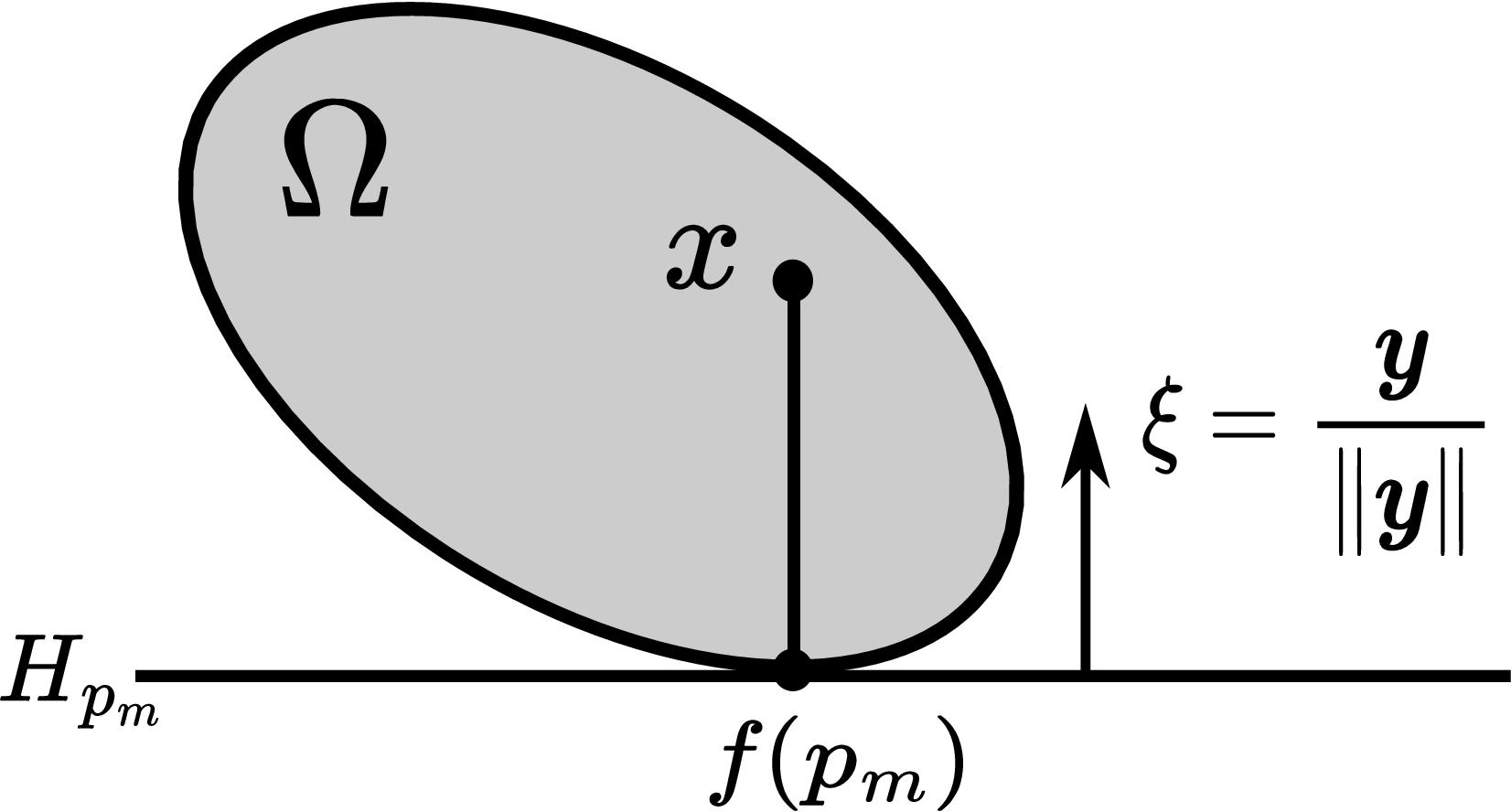}}
\end{tabular}  
\caption{The figure of closed convex domain $\Omega$  (Lemma~\ref{le:convex-domain}).}
\label{fig:convex-domain}
\end{figure}

Finally, we prove Theorem~\ref{thm:intro2}.

\begin{proof}[Proof of Theorem~\ref{thm:intro2}]
First, we assume that the total absolute curvature $\tau(M^n,f)$ is equal to $2$.
By Theorem~\ref{thm:intro1}, $M^n$ is homeomorphic to an $n$-sphere.
Since $\Sigma_f$ is non-empty and consists only of singular points of the first kind,
 $\bar f:\Sigma_f\to\R^{n+r}$ is an immersion
(cf.\ Remark~\ref{re:bar-f}).
Hence we can apply the original Chern--Lashof theorem
(Fact~\ref{fa:CL}) to $\bar f$.
Thus, Fact~\ref{fa:CL}-$(1)$ yields that the total absolute curvature of $\bar{f}$ is greater than or equal to $2$.
 Since $\tau(M^n,f) = 2$, we have
\begin{equation}\label{eq:tau-2}
 \int_{B } |G(p,\xi)|\, d \mu_B = 0, \quad
 (\,\tau(\Sigma_f, \bar{f}) =\,) 
 \frac{1}{\operatorname{vol} (S^{n+r-1})}
\int_{\bar{B} } |\bar{G}(q,\eta)|\, d \mu_{\bar{B}} = 2.
\end{equation}
Furthermore,  by Fact~\ref{fa:CL}-$(2)$ and $(3)$, 
$\Sigma_f$ is homeomorphic to an $(n-1)$-sphere and
the image $f(\Sigma_f)$ is a convex hypersurface embedded 
in an $n$-dimensional affine subspace $\mathcal{P}^n$.
Hence we obtain the condition (a).

Next, we prove the conditions (b) and (c).
Here, $f(\Sigma_f)$ is a convex hypersurface embedded 
in an $n$-dimensional affine subspace $\mathcal{P}^n$.
In addition, $G$ vanishes identically on $B_{\rm{reg}}$ by \eqref{eq:tau-2}.
By the argument used in the proof of Theorem~\ref{thm:intro1}-(3)
in Appendix~\ref{appendix:proof-theorem-a3}, 
it follows that $f(M^n_{\rm reg}) \subset \mathcal{P}^n$.
Since $\bar{f}:\Sigma_f \to \mathcal{P}^n$ is a convex hypersurface, 
$f(\Sigma_f)$ bounds a closed convex domain $\Omega$.
Now we prove that $f(M^n) = \Omega$.
First, we show that $f(M^n) \subset \Omega$.
In order to show this, we assume that $f(M^n) \not\subset \Omega$, 
leading to a contradiction.
We set a distance function $d_{\Omega}(p) := d(\Omega, f(p))$ $(p \in M^n)$.
By our assumption, 
the maximum $m$ of the function $d_{\Omega}$ is positive.
Let $p_{\rm{max}} \in M^n$ be a point attaining $m$.
If $p_{\rm{max}}$ is a singular point of $f$, 
then $f(p_{\rm{max}})$ is a point on the boundary of $\Omega$,
and hence $d_{\Omega}(p_{\rm{max}})=0$, 
which contradicts $d_{\Omega}(p_{\rm{max}})=m>0$.
So $p_{\rm{max}}$ is a regular point of $f$.
On the other hand, $p_{\rm{max}}$ is a critical point of $d_{\Omega}$.
Hence, $p_{\rm{max}}$ is a singular point of $f$, 
which contradicts that $p_{\rm{max}}$ is a regular point.
Therefore, $f(M^n)$ is a subset of $\Omega$.
Consequently, Lemma~\ref{le:convex-domain} yields that $f(M^n)$ coincides with $\Omega$.

Finally, we prove the converse.
Since $f(\Sigma_f)$ coincides with the boundary of the closed convex domain and $\Sigma_f$ is connected,
$\bar{f}$ is a convex hypersurface embedded in an $n$-dimensional affine subspace.
Therefore, the total absolute curvature of $\bar{f}$ is equal to $2$ by Fact~\ref{fa:CL}-$(3)$.
Since the image $f(M^n_{\rm{reg}})$ is contained in an $n$-dimensional affine subspace,
the Lipschitz-Killing curvature is equal to zero on $M^n_{\rm{reg}}$.
Consequently, the total absolute curvature $\tau(M^n,f)$ is equal to $2$.
\end{proof}

The following example gives a frontal satisfying the assumptions of
Theorem~\ref{thm:intro2} whose total absolute curvature is equal to $2$.

\medskip
\begin{example}\label{ex:surface-2}
		The map $f : S^n \to \R^{n+1}$ defined by
		$$
		f(x_1, \cdots , x_n , x_{n+1}) := (x_1, \cdots , x_n, 0) \qquad ((x_1, \cdots , x_{n+1}) \in S^n \subset \R^{n+1})
		$$
		is a frontal whose singular points are of the first kind
		( see Figure~\ref{fig:disk}; cf.\ \cite[Example 3.2]{Honda}). 
		 Thus $f$ is admissible.
		The singular set $\Sigma_f$ is an $(n-1)$-dimensional unit sphere $\{ (x_1, \cdots , x_n, x_{n+1}) \in S^n \mid x_{n+1} = 0 \}$.
		Since the image $f(S^n)$ is a closed unit $n$-ball in the $n$-dimensional subspace $\mathcal{P}^n = \{ (x_1,\cdots , x_n, 0) \in \R^{n+1} \}$, 
		the total absolute curvature $\int_B |G| \, d \mu_B =0$ holds.
		On the other hand, since the image $f(\Sigma_f)$ is a unit $(n-1)$-sphere, $\int_{\bar{B}} |\bar{G}| \, d \mu_{\bar{B}} = 2\operatorname{vol} (S^{n}) $ holds.
		Hence,  the regular contribution satisfies
		\(
		\int_B |G|\,d\mu_B=0.
		\)
		Therefore, $\tau(S^n,f)=2$.
		Moreover, the image $f(S^n)$ is a convex domain in $\mathcal{P}^n$ and the image $f(\Sigma_f)$ coincides with the boundary of $f(S^n)$.
\end{example}

\begin{figure}[htb]
\centering
 \begin{tabular}{c}
\resizebox{5cm}{!}{\includegraphics{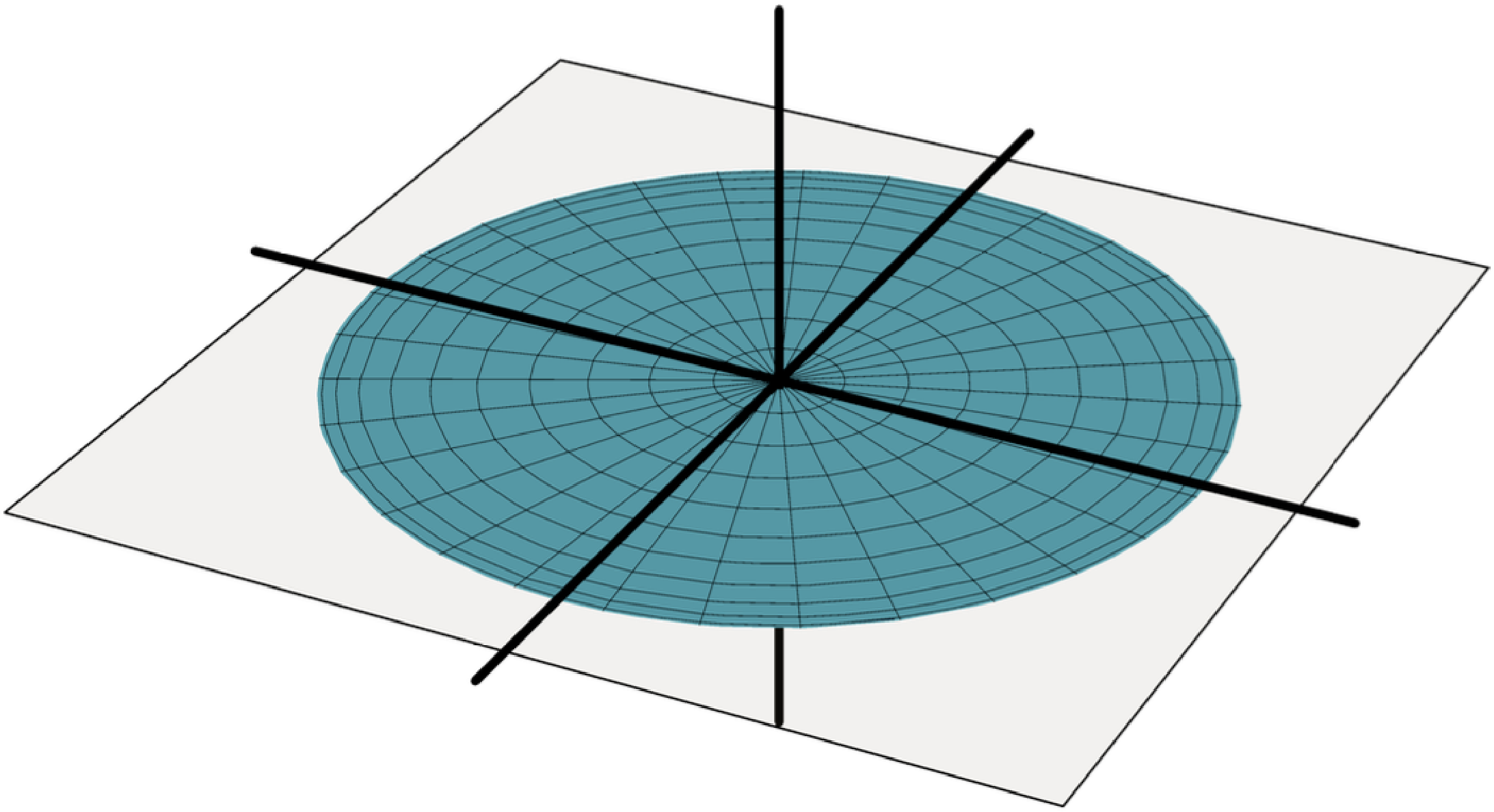}}
\end{tabular}  
\caption{The image $f(S^2)$ in $\R^3$  (Example~\ref{ex:surface-2}).}
\label{fig:disk}
\end{figure}

\section{ The $1$- and $2$-dimensional cases}
\label{sec:low-dimensional}

In this section, we exhibit the results obtained in the cases of $n=1,2$.

\subsection{In the case of $n=1$}

For a smooth map $\gamma: I \to \R^{r+1}$ defined on a non-empty interval $I$,
a point $c\in I$ is a singular point of $\gamma$ 
if and only if the derivative $\gamma'(c)$ vanishes, 
 where the prime $'$ denotes differentiation with respect to $t$.
A smooth map $\gamma: I \to \R^{r+1}$ whose regular set is dense in $I$ 
is a frontal if and only if there exists a smooth map $\vect{e}: I \to S^r$ 
such that $\gamma'(t)$ and $\vect{e}(t)$ are linearly dependent for each $t \in I$.
This map $\vect{e} (t)$ is said to be a {\it unit tangent vector field} along $\gamma$.
For a singular point $c \in I$ of a frontal $\gamma: I \to \R^{r+1}$,
if $\gamma''(c)$ does not vanish, the singular point $c$ is called a {\it generalized cusp} (cf. \cite{HNSUY}).
The notion of generalized cusps is introduced as {\it A-type} in \cite{HS}.
Generalized cusps are characterized by the non-degeneracy condition.

\medskip
\begin{proposition}\label{prop:generalized-cusp}
A singular point $c \in I$ of a frontal $\gamma: I \to \R^{r+1}$ is non-degenerate if and only if $c$ is a generalized cusp.
\end{proposition}
\medskip

\begin{proof}
Let $\lambda$ be the signed volume density function of $\gamma$.
The derivative $\lambda'(c)$ can be written as 
$$
\lambda'(c) = \det(\gamma'', E_1, \cdots , E_r) (c),
$$
where $\{ E_1, \cdots, E_r \}$ is an orthonormal frame of the normal bundle of $\gamma$.
Therefore, $\lambda'(c) \neq 0$ holds if and only if $\gamma''(c) \neq \vect{0}$.
\end{proof}

On the regular set of a frontal $\gamma: I \to \R^{r+1}$, the {\it curvature function} is defined as 
\begin{equation}\label{eq:curvature}
\kappa (t) := \frac{ \sqrt{ \|\gamma''\|^2 \| \gamma' \|^2 - (\gamma'' \cdot \gamma')^2 } } {\|\gamma'\|^3}.
\end{equation}
Since $\vect{e}(t) = \pm \gamma'(t) / \| \gamma'(t) \|$ holds on the regular set of $\gamma$, 
the curvature function $\kappa(t)$ can be written as $\kappa (t) =  \|\vect{e}'(t)\| / \|\gamma'(t)\|$.
Thus, $\kappa \, ds = \|\vect{e}' (t)\|\, dt$ holds, where $ds = \| \gamma'(t) \|\, dt$ is the arclength measure.
Although $\kappa$ is defined only on the regular set, the measure
$\kappa\,ds$ extends smoothly across non-degenerate singular points by
\[
\kappa\,ds=\|\vect e'(t)\|\,dt.
\]
In what follows, the integral $\int_{S^1}\kappa\,ds$ is understood in this sense.

If a frontal $\gamma:\R \to \R^{r+1}$ is a periodic map, $\gamma$ is called a {\it closed frontal}.
After rescaling the parameter, we may assume that $\gamma$ is $2\pi$-periodic.
Then, the domain of definition of $\gamma$ is regarded as $S^1 = \R / 2\pi \boldsymbol{Z}$.
A closed frontal $\gamma:S^1 \to \R^{r+1}$ is a co-orientable frontal if and only if the unit tangent vector field $\vect{e}:\R \to \R^{r+1}$ is a $2\pi$-periodic map, that is, 
$$
\vect{e} (t + 2 \pi) = \vect{e} (t) \quad (t \in \R).
$$

In the case of $n=1$, an admissible frontal $\gamma$ has no singular points of the second kind.
Thus, a frontal $\gamma$ is an admissible frontal if and only if singular points of $\gamma$ are non-degenerate.
Indeed, at a non-degenerate singular point, $\Sigma_\gamma$ is discrete, and hence
$T_c\Sigma_\gamma=\{0\}$, whereas $\ker d\gamma_c$ is one-dimensional.
Thus every non-degenerate singular point is of the first kind.
Then, the singular set $\Sigma_{\gamma}$ of a closed admissible frontal $\gamma:S^1\to \R^{r+1}$ is a discrete subset of $S^1$.
Hence, we may define the total absolute curvature for a closed frontal whose singular points are non-degenerate.
The total absolute curvature can be written 
using the curvature function $\kappa$ of $\gamma$.

\medskip
\begin{lemma}\label{lem:TAC-curve}
For a closed frontal $\gamma: S^1 \to \R^{r+1}$ whose singular points are generalized cusps,  
the total absolute curvature can be written as
$$
\tau(S^1,\gamma) = \frac{1}{\pi} \int_{S^1} \kappa \, ds + \# \Sigma_{\gamma}.
$$
\end{lemma}
\medskip

\begin{proof}
As shown in the proof of Theorem~\ref{thm:intro1}-(1),
we have
$$
\tau(S^1,\gamma) 
= \frac{1}{\operatorname{vol} (S^{r})} \int_{S^{r}\setminus Q} \# \operatorname{crit} (h_{\vect{w}},S^1_{\rm{reg}})\, d \mu_{S^{r}} 
+ \frac{1}{\operatorname{vol} (S^{r})} \int_{S^{r}\setminus Q} \# \operatorname{crit} (h_{\vect{w}},\Sigma_{\gamma})\, d \mu_{S^{r}}.
$$
Let $B$ be the unit normal bundle and $\nu$ be the canonical Gauss map of $\gamma$.
Then, 
$$
\frac{1}{\operatorname{vol} (S^{r})} \int_B |G|\, d \mu_B =\frac{1}{\operatorname{vol} (S^{r})} \int_{S^1} \int_{S^{r-1}} |\nu \cdot \vect{e}'|\, d \sigma dt 
= \frac{1}{\pi} \int_{S^1} \kappa \, ds
$$
holds, where $d\sigma$ is the volume element of $S^{r-1}$.
Therefore, we obtain
$$
\frac{1}{\operatorname{vol} (S^{r})} \int_{S^{r}\setminus Q} \# \operatorname{crit} (h_{\vect{w}},S^1_{\rm{reg}})\, d \mu_{S^{r}}  
=
\frac{1}{\pi} \int_{S^1} \kappa \, ds.
$$
Since all singular points are critical points of the height function $h_{\vect{w}}$ for each $\vect{w} \in S^{r}\setminus Q$,
we obtain 
$$
\frac{1}{\operatorname{vol} (S^{r})} \int_{S^{r}\setminus Q} \# \operatorname{crit} (h_{\vect{w}},\Sigma_{\gamma})\, d \mu_{S^{r}}
=
\# \Sigma_{\gamma}.
$$
Thus, we have the desired result.
\end{proof}

By Theorems~\ref{thm:intro1}, \ref{thm:intro2} and Lemma~\ref{lem:TAC-curve}, 
we have the following corollary in the case of $n=1$.

\medskip
\begin{corollary}\label{cor:fenchel}
For a co-orientable closed frontal $\gamma: S^1 \to \R^{r+1}$ whose singular points are generalized cusps,  
$$
\frac{1}{\pi} \int_{S^1} \kappa \, ds + \# \Sigma_{\gamma} \geq 2
$$
holds. Moreover, if $\Sigma_{\gamma}$ is not empty, then the total absolute curvature of $\gamma$ is equal to $2$ 
if and only if $\# \Sigma_{\gamma}$ is equal to $2$ and the image of $\gamma$ is a line segment.
\end{corollary}
\medskip

If the singular set $\Sigma_{\gamma}$ is empty, that is, $\gamma$ is a regular curve, then Corollary~\ref{cor:fenchel} is equivalent to Fenchel's theorem \cite{Fenchel1929, Fenchel1951}.
Furthermore, 
the frontal $\gamma$ is co-orientable if and only if 
$\#\Sigma_{\gamma}$ is an even integer. 
Indeed, since $\lambda$ changes its sign at each generalized cusp, 
the number of generalized cusps is even.
Hence, if the singular set $\Sigma_{\gamma}$ is not empty, $\# \Sigma_{\gamma}$ is greater than or equal to $2$.
In \cite{HTY}, the total absolute curvature of non-co-orientable closed frontals is investigated.

\subsection{In the case of $n=2$}
In this subsection, to compare with the result of Kossowski and Scherfner \cite{KS}, we consider frontals in $\R^3$.
A smooth map $f:M^2 \to \R^3$ whose regular set is dense in $M^2$ is a 
(co-orientable) frontal if and only if 
there exists a smooth map $N: M^2 \to S^2$ such that
$$
df_p (X) \cdot N(p) = 0 \quad (p \in M^2, \, X \in T_p M^2).
$$
Furthermore, a frontal $f:M^2 \to \R^3$ is called a {\it wave front} 
if $\mathcal{L} := (f,N) : M^2 \to \R^3 \times S^2$ is an immersion.
If all singular points are non-degenerate, 
the singular set $\Sigma_f$ is a disjoint union of embedded regular curves in $M^2$.
Therefore, a frontal $f:M^2 \to \R^3$ whose singular points are non-degenerate is an admissible frontal 
if and only if the set $C_f$ of singular points of the second kind is a discrete subset of $\Sigma_f$.

\medskip
\begin{remark}\label{re:admissible}
In \cite{USY}, if a singular point $p \in C_f$ of the second kind is not an accumulation point of $C_f$, then $p$ is said to be {\it admissible}. 
Namely, there are no singular points other than singular points of the first kind near admissible singular points of the second kind. 
Our admissible frontal of Definition~\ref{def:admissible} is a generalization of this definition.
\end{remark}
\medskip

For an admissible frontal $f:M^2 \to \R^3$, the singular set $\Sigma_f$ is a disjoint union of embedded regular curves in $M^2$.
We let $K$ be the Gaussian curvature of  $f$ and
define the curvature function $\kappa$ as in \eqref{eq:curvature} for $\bar{f} : \Sigma_f \to \R^3$.
Kossowski and Scherfner proved the following Chern-Lashof type theorem for wave fronts.

\medskip
\begin{fact}[\cite{KS}]\label{fa:ks}
Let $M^2$ be an orientable compact $2$-dimensional manifold 
and $f:M^2 \to \R^3$ be a co-orientable wave front such that all singular points are non-degenerate 
and the set of singular points of the second kind is a discrete subset of the singular set $\Sigma_f$.
Then,
$$
\int_{M^2} |K|\, dV + 2 \int_{\Sigma_f} \kappa \, ds  \geq 2 \pi ( 2 + 2 g (M^2))
$$
holds, where $g(M^2)$ is the genus of $M^2$.
\end{fact}
\medskip

By Theorems A and B, we have the following corollary in the case of $n=2$ and $r=1$.

\medskip
\begin{corollary}\label{cor:surface}
Let $M^2$ be an orientable compact $2$-dimensional manifold and 
$f: M^2 \to \R^3$ be a co-orientable frontal such that all singular points are non-degenerate and 
the set of singular points of the second kind is a discrete subset of the singular set $\Sigma_f$.
Then, 
$$
\int_{M^2} |K|\, dV + 2 \int_{\Sigma_f} \kappa \, ds  \geq 2 \pi ( 2 + 2 g (M^2))
$$
holds, where $g(M^2)$ is the genus of $M^2$.
In particular, 
\begin{equation}\label{eq:ineq-2}
\int_{M^2} |K|\, dV + 2 \int_{\Sigma_f} \kappa \, ds  \geq 4 \pi
\end{equation}
holds.
Moreover, suppose that the singular set $\Sigma_f$ is not empty and singular points are of the first kind. 
If equality of \eqref{eq:ineq-2} holds then $M^2$ is homeomorphic to a $2$-dimensional sphere, and the image $f(M^2)$ is a closed convex domain of an affine plane of $\R^3$.
\end{corollary}
\medskip

Corollary~\ref{cor:surface} gives a generalization of Fact~\ref{fa:ks} to frontals.
Furthermore, 
our generalization to frontals clarifies the relationship between the equality case and convexity.

\appendix
\renewcommand{\thesection}{\Roman{section}}

\refstepcounter{section}
\section*{Appendix \thesection. Proof of Theorem~\ref{thm:intro1}-(3)}
\addcontentsline{toc}{section}{Appendix \thesection. Proof of Theorem~\ref{thm:intro1}-(3)}
\label{appendix:proof-theorem-a3}

In this appendix, we give a proof of Theorem~\ref{thm:intro1}-(3).

\medskip
If $r=1$, the assertion is trivial.
Thus, we assume that $r \geq 2$.

First, we show that $f(M^n)$  is contained in 
an $(n+r-1)$-dimensional affine subspace of $\R^{n+r}$. 
In order to show this, we assume that 
$f(M^n)$ is not contained in any $(n+r-1)$-dimensional affine subspace of $\R^{n+r}$, leading to a contradiction.

First, we prove by contradiction that $G(p,\xi)$ vanishes identically on $B_{\rm reg}$.  
(Our proof is a modification of that of \cite[Lemma~1]{CL1}.)  
Suppose that there exists $(p,\xi) \in B_{\rm reg}$ such that $G(p,\xi) \ne 0$.
	We take an orthonormal frame $\{ \xi_1, \cdots , \xi_{r-1}, \xi \}$ of $\Pi^{\bot}(p)$.
	Let $\xi_\theta$ be a unit vector defined by
	$$
	\xi_\theta  := \cos \theta \, \xi + \sin \theta \, \xi_1 \quad (\theta \in [0,2\pi)).
	$$
	Then, we define a function $\rho(\theta):=G(p,\xi_\theta)$.
	Since $G(p,\xi) \neq 0$, the function $\rho(\theta)$ does not vanish identically.
	Let $H_\theta$ be the hyperplane through $f(p)$ perpendicular to $\xi_\theta$.
	We define $\operatorname{pr}: \R^{n+r} \to \mathcal{R}$ as the orthogonal projection onto the subspace $\mathcal{R}$ spanned by $\xi$ and $\xi_1$.
	Then, $\operatorname{pr} (f(M^n)-f(p))$ is not contained in any line in the plane $\mathcal{R}$.
	Hence, there exist two points $q_1$ and $q_2$ in $M^n$ 
	such that two non-zero vectors $\operatorname{pr} (f(q_1) - f(p))$ and $\operatorname{pr} (f(q_2) - f(p))$ are not parallel.
	Here, we show that there exists $\theta_3$ such that $\rho(\theta_3) \neq 0$ and 
	$f(q_1)$ and $f(q_2)$ lie on different sides of the hyperplane $H_{\theta_3}$. 
	We set a function 
	$$\delta(\theta) := ((f(q_1) - f(p)) \cdot \xi_\theta)((f(q_2) - f(p)) \cdot \xi_\theta).$$
	In this setting,
	two points $f(q_1)$ and $f(q_2)$ lie on different sides of the hyperplane $H_\theta$ 
	if and only if $\delta(\theta) < 0$.
	We denote by $\vect{e}_1$ and $\vect{e}_2$ the unit vectors defined by
	$$
	\vect{e}_1  := \frac{\operatorname{pr}(f(q_1) - f(p))}{\|\operatorname{pr}(f(q_1) - f(p))\|}, \quad
	\vect{e}_2  := \frac{\operatorname{pr}(f(q_2) - f(p))}{\|\operatorname{pr}(f(q_2) - f(p))\|}.
	$$
	Then, we write $\vect{e}_1$ and $\vect{e}_2$ as
	$$
	\vect{e}_1 = \cos \alpha \, \xi + \sin \alpha \, \xi_1, \quad
	\vect{e}_2 = \cos \beta \, \xi + \sin \beta \, \xi_1 \quad (\alpha, \beta \in [0,2\pi)).
	$$
	As a result, $\delta(\theta)$ can be written as
	$$
	\delta(\theta) = \|\operatorname{pr}(f(q_1) - f(p))\| \|\operatorname{pr}(f(q_2) - f(p))\| 
	\cos (\theta - \alpha) \cos (\theta - \beta).
	$$
	Since $\vect{e}_1$ and $\vect{e}_2$ are not parallel, $\alpha - \beta \neq 0,\pm \pi$.
	Thus, the set $\{ \theta \in [0,2\pi) \mid \delta(\theta) < 0 \}$ is a non-empty open subset.
	On the other hand, for $\rho(\theta)$, 
	we have 
	$$
	\rho(\theta) = \det A_{\cos \theta \, \xi + \sin \theta \, \xi_1} = \det( \cos \theta \, A_{\xi} + \sin \theta \, A_{\xi_1} ). 
	$$ 
	Hence, $\rho(\theta)$ is a polynomial of $\cos \theta$ and $\sin \theta$.
	Namely, $\rho(\theta)$ is a real analytic function of $\theta$.
	Since $\rho(\theta)$ does not vanish identically, 
	the set $\{ \theta \in [0,2\pi) \mid \rho(\theta) \neq 0 \}$ is dense in $[0,2\pi)$.
	Therefore, the set $\{ \theta \in [0,2\pi) \mid \delta(\theta) < 0, \rho(\theta) \neq 0 \}$ is not empty.
	We choose $\theta_3$ from the set.
	Next, we set a function $\Delta: B_{\rm reg} \to \R$ by 
	$$
	\Delta(p',\xi')  := ((f(q_1) - f(p')) \cdot \xi')((f(q_2) - f(p')) \cdot \xi').
	$$
	Since $\Delta(p,\xi_{\theta_3}) < 0$, 
	there exists an open neighborhood $W$ of $(p,\xi_{\theta_3})$ in $B_{\rm reg}$ 
	such that $\Delta(p',\xi') < 0$ holds for each $(p',\xi') \in W$.
	By the inverse function theorem, 
	by choosing $W$ sufficiently small if necessary
	the map $\nu |_W: W \to \nu(W)$ is a diffeomorphism onto its image.
	For $\xi'\in \nu(W)\cap Q^c$, 
	let $(p',\xi')\in W$ be the point with $\nu(p',\xi')=\xi'$. 
	Then $p'$ is a Morse critical point of $h_{\xi'}$.
	Moreover, since $f(q_1)$ and $f(q_2)$ lie on different sides of the hyperplane
	through $f(p')$ perpendicular to $\xi'$, the value $h_{\xi'}(p')$ is neither
	the maximum nor the minimum of $h_{\xi'}$ on $M^n$.
	Thus $h_{\xi'}$ has at least three critical points.
	Since $\nu(W)$ has positive measure in $S^{n+r-1}$,
	the total absolute curvature $\tau(M^n,f)$ is greater than $2$ 
	 (cf.\ \eqref{eq:ave-tau} ).
	This contradicts $\tau(M^n,f) = 2$.

Thus, $G(p,\xi) $ is identically zero on $B_{\rm reg}$.
Then, we obtain
$$
 \int_{B } |G|\, d \mu_B = 0,\quad
 \frac{1}{\operatorname{vol} (S^{n+r-1})}
\int_{\bar{B} } |\bar{G}|\, d \mu_{\bar{B}} = 2.
$$
Hence, there exists $(q,\eta) \in \bar{B}$ such that $\bar{G}(q,\eta) \neq 0$.
Then we can prove that $f(\Sigma_f)$  is contained in
an $(n+r-1)$-dimensional affine subspace $\mathcal{P}^{n+r-1}$ of $\R^{n+r}$ 
 by an argument similar to the one above (cf.\ \cite[Lemma~1]{CL1}). 
Indeed, if $f(\Sigma_f)$ is not contained in any $(n+r-1)$-dimensional affine subspace of $\R^{n+r}$, 
then there would exist an open subset $W'$ of $\bar{B}$ such that  
the height function $h_{\eta'}$ has at least three critical points for each $(q',\eta') \in W'$, 
which leads to a contradiction.
Therefore, $f(\Sigma_f)$  is contained in an $(n+r-1)$-dimensional affine subspace $\mathcal{P}^{n+r-1}$ of $\R^{n+r}$.
As a result, it suffices to show that $f(M^n_{\rm{reg}})$ belongs to $\mathcal{P}^{n+r-1}$.

In order to show this, we assume that $f(M^n_{\rm{reg}})$ is not contained in $\mathcal{P}^{n+r-1}$, leading to a contradiction.
By a parallel translation if necessary, we may assume that $\mathcal{P}^{n+r-1}$ passes through the origin.
Then there exists a unit vector $\vect{a} \in S^{n+r-1}$ such that
$
\mathcal{P}^{n+r-1} = \{ \vect{v} \in \R^{n+r} \mid \vect{v} \cdot \vect{a}  =0 \}.
$
Since the image $f(M^n_{\rm{reg}})$ is not contained in $\mathcal{P}^{n+r-1}$, 
there exists $q_0 \in M^n_{\rm{reg}}$ such that $h_{\vect{a}} (q_0) = f(q_0) \cdot \vect{a}$ is greater than zero.
For $\vect{b} \in S^{n+r-1}$ such that  $\vect{a}$ and $\vect{b}$ are linearly independent,
we obtain 
$$
h_{\vect{b}} (q_0) = f(q_0) \cdot \vect{b} = f(q_0) \cdot (\vect{b}-\vect{a}) + f(q_0) \cdot \vect{a}.
$$
Since $f(\Sigma_f) \subset \mathcal{P}^{n+r-1}$, we have $f(p) \cdot \vect{a} = 0$ for each $p \in \Sigma_f$.
Hence, 
$$
h_{\vect{b}} (p) = f(p) \cdot \vect{b} =
f(p) \cdot (\vect{b}-\vect{a}) \quad (p \in \Sigma_f)
$$
holds. 
Therefore, we have
$$
h_{\vect{b}} (q_0) - h_{\vect{b}} (p) = f(q_0) \cdot \vect{a} - (f(p)-f(q_0)) \cdot (\vect{b}-\vect{a}) \quad (p \in \Sigma_f).
$$
By the Cauchy-Schwarz inequality, it holds that
$$
f(q_0) \cdot \vect{a} - (f(p)-f(q_0)) \cdot (\vect{b}-\vect{a}) \geq f(q_0) \cdot \vect{a} - \| f(p)-f(q_0) \|  \| \vect{b}-\vect{a} \|  \quad (p \in \Sigma_f).
$$
If
$
f(q_0) \cdot \vect{a} - \| f(p)-f(q_0) \|  \| \vect{b}-\vect{a} \|  
$
is greater than zero, 
we have $h_{\vect{b}} (q_0) > h_{\vect{b}} (p)$.
Now, we define a function $\varphi$ on 
$\Sigma_f$ by $\varphi(p) :=\| f(p)-f(q_0) \|$.
Since $f(q_0)\notin\mathcal P^{n+r-1}$ while
$f(\Sigma_f)\subset\mathcal P^{n+r-1}$, we have
$\varphi(p)>0$ for all $p\in\Sigma_f$.
We denote by $m\, (>0)$ the maximum value of $\varphi$.
We set $Z \subset S^{n+r-1}$ as
$$
Z  := \left\{ \vect{b} \in S^{n+r-1} \,\middle|\, \frac{f(q_0) \cdot \vect{a}}{m} > \| \vect{b}-\vect{a} \| \right\}.
$$
Then,  for each $\vect{b} \in Z$, there exists a maximum point of the height function $h_{\vect{b}} $ on $M^n_{\rm{reg}}$.
We denote by $Q^c = S^{n+r-1} \setminus Q$ the complement of $Q$.
Since $Z$ is a non-empty open subset of $S^{n+r-1}$ and $Q$ has measure zero,
we have $Z\cap Q^c\neq\emptyset$.
For a given $\vect{b}_0 \in Z \cap Q^c$, 
we let $q_{\rm{max}} \in M^n_{\rm{reg}}$ be a point which attains the maximum of $h_{\vect{b}_0}$. 
Then the maximum point is a Morse critical point, that is, $\vect{b}_0$ is a regular value of $\nu$.
In other words, $G(q_{\rm{max}} , \vect{b}_0)$ is not zero.
This contradicts the assumption that $G(p,\xi) = 0$ for each $(p,\xi) \in B_{\rm{reg}}$.
As a result, the image $f(M^n)$  is contained in 
an $(n+r-1)$-dimensional affine subspace of $\R^{n+r}$.

We denote by $f':M^n \to \mathcal{P}^{n+r-1}$ the restriction of $f$ to $\mathcal{P}^{n+r-1}$.
By parallel translation if necessary, we may assume that $\mathcal{P}^{n+r-1}$ coincides with $(n+r-1)$-dimensional Euclidean space $\R^{n+r-1}$.
Then, 
$$
\tau(M^n,f') = \frac{1}{\operatorname{vol} (S^{n+r-2})} \int_{S^{n+r-2} \setminus Q'} \# \operatorname{crit} (h_{\vect{w}},M^n)\, d \mu_{S^{n+r-2}}
$$
holds by \eqref{eq:ave-tau}.
Here, \(Q'\) is the subset of the unit sphere $S^{n+r-2}$ defined as in \eqref{eq:setQ}
for the frontal \(f'\).
Let $\xi_0$ be a unit vector such that $\xi_0$ is perpendicular to $\R^{n+r-1}$.
Then, for almost every $\vect{w} \in S^{n+r-2} \setminus Q'$, 
there exists $s \neq 0$ and $t$ such that $\vect{v} = s\, \vect{w} + t\, \xi_0 \in S^{n+r-1} \setminus Q$.
Since $\tau(M^n,f) = 2$, the height function $h_{\vect{v}}$ has exactly two critical points.
On the other hand, by $h_{\vect{v}} = f \cdot (s\, \vect{w} + t\, \xi_0)$,
we have $h_{\vect{v}} = s h_{\vect{w}} + t h_{\xi_0}$.
Furthermore, because $f'(M^n) \subset \R^{n+r-1}$, the height function $h_{\xi_0}$ vanishes identically, so $h_{\vect{v}} = s h_{\vect{w}}$ holds.
Thus, the height function $h_{\vect{v}}$ has the same critical points as $h_{\vect{w}}$.
Therefore, the height function $h_{\vect{w}}$ has exactly two critical points for almost every $\vect{w} \in S^{n+r-2} \setminus Q'$,
in other words, $\tau(M^n,f') = 2$.

As a result,
viewing $f$ as a frontal in the affine subspace $\mathcal{P}^{n+r-1}$, 
the total absolute curvature is unchanged 
under such a reduction of the codimension.
By repeating the process inductively, we can show that $f(M^n)$  is contained in 
an $(n+1)$-dimensional affine subspace.

\medskip
\begin{remark}\label{rem:n-subspace}
Under the assumption of $\tau(M^n,f) = 2$,
one cannot in general conclude that $f(M^n)$
is contained in an $n$-dimensional affine subspace. 
The argument above reduces the
ambient dimension by using
$\xi_\theta$
defined in a two-dimensional subspace of the normal space. 
However, after the dimension has been
reduced to $(n+1)$, 
the normal space is one-dimensional, 
so such a vector $\xi_\theta$ is no longer available. 
Hence the argument cannot be used to show that $G$ vanishes identically on $B_{\rm reg}$, 
and the dimension reduction stops at $(n+1)$.
\end{remark}

\medskip
\begin{acknowledgements}
The author is deeply grateful to Professor Atsufumi Honda, 
his advisor, for the invaluable mentorship, advice, encouragement, support and leading him to this result.
He would also like to thank Professors Kazuyuki Enomoto, Masaaki Umehara, and Kotaro Yamada for their insightful comments and communications, which have been instrumental in shaping this work.
He is also grateful to the anonymous reviewer for the constructive feedback that helped improve the paper.
This work was supported in part by JST SPRING, Grant Number JPMJSP2178.
\end{acknowledgements}

%
%
%
%

\end{document}